\documentclass[12pt,english,a4paper]{smfart}

\usepackage[T1]{fontenc}
\usepackage{lmodern}
\usepackage{smfthm}
\usepackage[headings]{fullpage}
\usepackage{amssymb}

\newcommand{\Qp}{\mathbf{Q}_p}
\newcommand{\Zp}{\mathbf{Z}_p}
\newcommand{\Cp}{\mathbf{C}_p}
\newcommand{\Kalg}{\overline{K}}
\newcommand{\ZZ}{\mathbf{Z}}
\newcommand{\QQ}{\mathbf{Q}}
\newcommand{\Qbar}{\overline{\mathbf{Q}}}
\newcommand{\OO}{\mathcal{O}}
\newcommand{\bigO}{\mathrm{O}}
\newcommand{\MM}{\mathfrak{m}}
\newcommand{\End}{\operatorname{End}}
\newcommand{\Id}{\operatorname{Id}}
\newcommand{\Hom}{\operatorname{Hom}}
\newcommand{\Gal}{\operatorname{Gal}}
\newcommand{\GL}{\operatorname{GL}}
\newcommand{\Gm}{\mathbf{G}_m}
\newcommand{\Dyn}{\mathcal{D}}
\newcommand{\Tors}{\operatorname{Tors}}
\newcommand{\preper}{\operatorname{Preper}}
\newcommand{\Log}{\mathrm{Log}}
\newcommand{\dcroc}[1]{[\![ #1 ]\!]}

\renewcommand{\geq}{\geqslant}
\renewcommand{\leq}{\leqslant} 

\author{Laurent Berger}
\address{UMPA de l'ENS de Lyon \\
UMR 5669 du CNRS}
\email{laurent.berger@ens-lyon.fr}
\urladdr{perso.ens-lyon.fr/laurent.berger/}

\date{\today}

\title[Rigidity and unlikely intersections for formal groups]{Rigidity and unlikely intersections \\ for formal groups}

\begin{document}

\begin{abstract}
Let $K$ be a $p$-adic field and let $F$ and $G$ be two formal groups over $\OO_K$. We prove that if $F$ and $G$ have infinitely many torsion points in common, then $F=G$. This follows from a rigidity result: any bounded power series that sends infinitely many torsion points of $F$ to torsion points of $F$ is an endomorphism of $F$.
\end{abstract}

\subjclass{11S31 (11F80; 11S82; 13J05; 37P35)}

\keywords{Rigidity; unlikely intersections; formal groups; $p$-adic dynamical systems; preperiodic points}

\maketitle

\tableofcontents

\setlength{\baselineskip}{18pt}

\section*{Introduction}

Let $K$ be a finite extension of $\Qp$ (or, more generally, a finite extension of $W(k)[1/p]$ where $k$ is a perfect field of characteristic $p$). Let $\Kalg$ be an algebraic closure of $K$ and let $\Cp$ be the $p$-adic completion of $\Kalg$. Let $\OO_K$ denote the ring of integers of $K$, and let $F(X,Y) = X \oplus Y \in \OO_K \dcroc{X,Y}$ be a formal group law over $\OO_K$. Let $\Tors(F)$ be the set of torsion points of $F$ in $\MM_{\Cp}= \{ z \in \Cp,\ |z|_p<1\}$. The question that motivates this paper is: to what extent is a formal group $F$ determined by $\Tors(F)$?  Our main result is an ``unlikely intersections'' result. 

\begin{enonce*}{Theorem A}
If $F$ and $G$ are two formal groups over $\OO_K$ and if $\Tors(F) \cap \Tors(G)$ is infinite, then $F=G$.
\end{enonce*}

If $n \geq 2$ and $[n](X)$ denotes the multiplication by $n$ map on $F$, then $\Tors(F)$ is also the set of preperiodic points of $[n](X)$ in $\MM_{\Cp}$. We can therefore think of $\Tors(F)$ as the set $\preper(F)$ of preperiodic points of a $p$-adic dynamical system attached to $F$. Theorem A then becomes a statement about preperiodic points of certain dynamical systems.

Some analogues of theorem A are known in other contexts. For example, if two elliptic curves over $\Qbar$ have infinitely many torsion points in common (in a suitable sense), then they are isomorphic (Bogomolov and Tschinkel, see \S 4 of \cite{BT07}). In another context, if $f$ and $g$ are two rational fractions of degree at least $2$ with coefficients in the complex numbers, and if $\preper(f) \cap \preper(g)$ is infinite, then $\preper(f)=\preper(g)$ (Baker and DeMarco, theorem 1.2 of \cite{BDM11}). In this case, $f$ and $g$ have the same Julia set (corollary 1.3 of ibid.). One can then show that, if $f$ and $g$ are polynomials of the same degree, then in most cases they are equal up to a linear symmetry that preserves their common Julia set (see for instance \cite{BE87} and \cite{SS95}). 

Our proof of theorem A relies on a rigidity result for formal groups. We say that a subset $Z \subset \MM_{\Cp}^d$ is Zariski dense  in $\MM_{\Cp}^d$ if every power series $h(X_1,\hdots,X_d) \in \OO_K \dcroc{X_1,\hdots,X_d}$ that vanishes on $Z$ is necessarily equal to zero. For example, if $d=1$ then $Z \subset \MM_{\Cp}$ is Zariski dense in $\MM_{\Cp}$ if and only if it is infinite.

\begin{enonce*}{Theorem B}
If $F$ is a formal group over $\OO_K$ and if $h(X) \in X \cdot \OO_K \dcroc{X}$ is such that $h(z) \in \Tors(F)$ for infinitely many $z \in \Tors(F)$, then $h \in \End(F)$.

More generally, if $h(X_1,\hdots,X_d) \in \OO_K \dcroc{X_1,\hdots,X_d}$ is such that $h(0)=0$ and $h(z) \in \Tors(F)$ for all $z$ in a subset of $\Tors(F)^d$ that is  Zariski dense in $\MM_{\Cp}^d$, then there exists $h_1,\hdots,h_d \in \End(F)$ such that $h=h_1(X_1) \oplus \cdots \oplus h_d(X_d)$.
\end{enonce*}

This theorem generalizes corollary 4.2 of Hida's \cite{H14}, which concerns the case $F=\Gm$. Our proof uses ideas coming from the theory of $p$-adic dynamical systems (developed in large part by Lubin, see \cite{L94}) rather than the ``special subvarieties'' argument of Hida (which is in the spirit of Chai's \cite{C08}). Other kinds of ``unlikely intersections'' results for certain formal groups can be found in \cite{S16}.

\vspace{\baselineskip}\noindent\textbf{Acknowledgements.} This paper is motivated by Holly Krieger's talk ``A dynamical approach to common torsion points'' at the CNTA XV meeting. I am grateful to Holly Krieger for useful discussions and to the organizers of CNTA XV for inviting me.

\section{Formal groups}
\label{fgsec}

For the basic definitions and results about formal groups that we need, we refer for instance to Lubin's \cite{L64,L67}. Let $F(X,Y) = X \oplus Y \in \OO_K \dcroc{X,Y}$ be a formal group law over $\OO_K$. If $n \in \ZZ$, let $[n](X)$ denote the multiplication by $n$ map on $F$. More generally, if $a \in \OO_K$, let $[a](X)$ be the unique endomorphism of $F$ such that $[a]'(0)=a$ if it exists. Let $\Tors(F)$ be the set of torsion points of $F$. If $F$ is of finite height, then $\Tors(F)$ is infinite, while if $F$ is of infinite height, then $\Tors(F)$ is finite and our results are vacuous. We therefore assume from now on that $F$ is of finite height $h$.

Let $T_p F = \varprojlim_n F[p^n]$ be the Tate module of $F$. If $F$ is of height $h$, then $T_p F$ is a free $\Zp$-module of rank $h$, equipped with an action of $\Gal(\Kalg/K)$. If we choose a basis of $T_p F$, this gives a Galois representation $\rho_F : \Gal(\Kalg/K) \to \GL_h(\Qp)$. Let $E$ be the fraction field of $\End(F)$. It is a finite extension of $\Qp$ whose degree $e$ divides $h$ (theorem 2.3.2 of \cite{L64}), so that we can view $\GL_{h/e}(E)$ as a subgroup of $\GL_h(\Qp)$.

\begin{theo}
\label{serresurj}
The image of $\rho_F$ is  an open subgroup of a conjugate of $\GL_{h/e}(E)$.
\end{theo}

\begin{proof}
This is an unpublished theorem of Serre (see however the remark after theorem 5 on page 130 of \cite{S67}), which is proved in \cite{S73} (see theorem 3 on page 168 and the remark that follows).
\end{proof}

\begin{coro}
\label{scalopen}
The image of $\rho_F$ contains an open subgroup of $\Zp^\times \cdot \Id$. 
\end{coro}

If $\sigma \in \Gal(\Kalg/K)$ is such that $\rho_F(\sigma) = a \cdot \Id$, then $\sigma(z)=[a](z)$ for all $z \in \Tors(F)$.

\section{$p$-adic dynamical systems}
\label{dssec}

In this {\S}, we prove a number of results about power series that commute under composition (sometimes also called permutable power series). These results are all inspired by Lubin's theory of $p$-adic dynamical systems (see \cite{L94}).

A power series $h(X) \in X \cdot K \dcroc{X}$ is said to be stable if $h'(0)$ is neither $0$ nor a root of unity. If $h'(0) \neq 0$, then there exists a unique power series $h^{\circ -1}(X) \in X \cdot K \dcroc{X}$ such that $h \circ h^{\circ -1} = h^{\circ -1} \circ h = X$. If in addition $h(X) \in X \cdot \OO_K \dcroc{X}$ and $h'(0) \in \OO_K^\times$, then $h^{\circ -1}(X) \in  X \cdot \OO_K \dcroc{X}$.

\begin{theo}
\label{comgen}
Let $u(X) \in X \cdot K \dcroc{X}$ be a stable power series. 

A power series $h(X_1,\hdots,X_d) \in K \dcroc{X_1,\hdots,X_d}$ such that $h(0)=0$ and such that $h \circ u = u \circ h$ is determined by $\{dh/dX_i(0) \}_{1 \leq i \leq d}$.
\end{theo}

\begin{proof}
Suppose that $h^{(1)}$ and $h^{(2)}$ are two such power series, and that they coincide in degrees $\leq m$. Let $h_m$ be the sum of the terms of $h^{(i)}$ of total degree $\leq m$. We have $h^{(i)} = h_m + r^{(i)}$ with $r^{(i)}$ of degree $\geq m+1$, and  
\[ \begin{cases} (h_m+r^{(i)}) \circ u = h_m \circ u + r^{(i)} \circ u \equiv h_m \circ u + u'(0)^{m+1} r^{(i)}  \bmod{\deg\ (m+2)}, \\ 
u \circ (h_m + r^{(i)}) \equiv u \circ h_m + r^{(i)} u'(h_m) \equiv u \circ h_m + r^{(i)} u'(0)  \bmod{\deg\ (m+2)}. \end{cases} \]
Since $u'(0)^m \neq 1$, the fact that $h^{(i)} \circ u = u \circ h^{(i)}$ implies that 
\[ r^{(i)} \equiv \frac{h_m \circ u - u \circ h_m}{u'(0)-u'(0)^{m+1}} \bmod{\deg\ (m+2)}. \] 
If $h^{(1)}$ and $h^{(2)}$ coincide in degrees $\leq m$, they therefore have to coincide in degrees $\leq m+1$. This implies the theorem by induction on $m$.
\end{proof}

Let us say that an endomorphism of a formal group is stable if the corresponding power series is stable.

\begin{coro}
\label{isendo}
Let $F$ be a formal group and let $u$ be a stable endomorphism of $F$. If $h(X) \in X \cdot \OO_K \dcroc{X}$ is such that $h \circ u = u \circ h$, then $h$ is an endomorphism of $F$.
\end{coro}

\begin{proof}
The power series $F \circ h$ and $h \circ F$ both commute with $u$, and have the same derivatives at $0$, so that $F \circ h = h \circ F$ by theorem \ref{comgen}.
\end{proof}

\begin{coro}
\label{comfg}
If $u$ is a stable endomorphism of a formal group and if $h(X_1,\hdots,X_d) \in \OO_K \dcroc{X_1,\hdots,X_d}$ is such that $h(0)=0$ and $h \circ u = u \circ h$, then there exists $a_1\hdots,a_d \in \OO_K$ such that $h(X_1,\hdots,X_d) = [a_1](X_1) \oplus \cdots \oplus [a_d](X_d)$.
\end{coro}

\begin{proof}
Let $h_i(X)$ be the power series $h$ evaluated at $X_i=X$ and $X_k=0$ for $k \neq i$. We have $h_i \circ u = u \circ h_i$ and hence by corollary \ref{isendo}, $h_i(X)=[a_i](X)$ where $a_i = h_i'(0) \in \OO_K$. The two power series $h(X_1,\hdots,X_d)$ and $[a_1](X_1) \oplus \cdots \oplus [a_d](X_d)$ commute with $u$ and have the same derivatives at $0$, so that they are equal by theorem \ref{comgen}.
\end{proof}

\section{Rigidity and unlikely instersections}
\label{mainsec}

We first recall and prove theorem B.

\begin{theo}
\label{thAbis}
If $F$ is a formal group over $\OO_K$ and if $h(X_1,\hdots,X_d) \in \OO_K \dcroc{X_1,\hdots,X_d}$ is such that $h(0)=0$ and $h(z) \in \Tors(F)$ for all $z$ in a subset $Z$ of $\Tors(F)^d$ that is  Zariski dense in $\MM_{\Cp}^d$, then there exists $h_1,\hdots,h_d \in \End(F)$ such that $h=h_1(X_1) \oplus \cdots \oplus h_d(X_d)$.
\end{theo}

\begin{proof}
Since $\Tors(F)$ is infinite, $F$ is of finite height. By corollary \ref{scalopen}, there exists $\sigma \in \Gal(\Kalg/K)$ and a stable endomorphism $u$ of $F$ such that $\sigma(z) = u(z)$ for all $z \in \Tors(F)$. If $z \in Z$, then we have $\sigma(h(z))=u(h(z))$ as well as $\sigma(h(z)) = h (\sigma(z)) = h(u(z))$. The power series $u \circ h - h \circ u$ therefore vanishes on $Z$. Since $Z$ is Zariski dense in $\MM_{\Cp}^d$, we have $u \circ h = h \circ u$. The theorem now follows from corollary \ref{comfg}.
\end{proof}

\begin{rema}
\label{zardens}
If $Y_1,\hdots,Y_d$ are infinite subsets of $\Tors(F)$, then $Y_1 \times \cdots \times Y_d$ is Zariski dense in $\MM_{\Cp}^d$.
\end{rema}

We now recall and prove theorem A. 

\begin{theo}
\label{thBbis}
If $F$ and $G$ are two formal groups over $\OO_K$ and if $\Tors(F) \cap \Tors(G)$ is infinite, then $F=G$.
\end{theo}

\begin{proof}
By corollary \ref{scalopen}, there exists an element $\sigma \in \Gal(\Kalg/K)$ and a stable endomorphism $u$ of $F$ such that $\sigma(z) = u(z)$ for all $z \in \Tors(F)$. The set $\Lambda = \Tors(F) \cap \Tors(G)$ is stable under the action of $\Gal(\Kalg/K)$. If $z \in \Lambda$, we therefore have $\sigma(z) \in \Lambda$ and hence $u(z) \in \Tors(G)$ for all $z \in \Lambda$, since $u(z) = \sigma(z)$. By theorem B applied to $G$, we get that $u \in \End(G)$. The power series $F$ and $G$ commute with $u$ and have the same linear terms, hence $F=G$ by theorem \ref{comgen}.
\end{proof}

\section{Generalizations and perspectives}
\label{beyondsec}

\subsection{Universal bounds}
In \S 4 of \cite{BT07}, Bogomolov and Tschinkel prove that two nonisomorphic elliptic curves over $\Qbar$ have only finitely many torsion points in common. In \cite{BFT}, the authors raise the question of the existence of a universal bound for the maximum number of torsion points that two nonisomorphic elliptic curves over $\overline{\QQ}$ (or even over the complex numbers) can share. The same kind of question is raised, for preperiodic points of rational fractions, in the forthcoming paper \cite{DKY}.

The following proposition shows that in our situation, there is no  straightforward refinement of theorem A.

\begin{prop}
\label{manytors}
For all $m \geq 1$, there exists a formal group $F$ over $\Zp$, of height $1$, such that $F$ is not isomorphic to $\Gm$ but such that $\Tors(F) \cap \Tors(\Gm)$ contains at least $m$ points.
\end{prop}

\begin{proof}
Take $n \geq 1$ and let $q(X)=(1+X)^p-1$ and $u(X)=1+((1+X)^{p^n}-1)/X$ and $f(X)=u(X)q(X)$. We have $f(X)=p(1+p^n)X + \bigO(X^2)$ and $f(X) \equiv X^p \bmod{p}$. 

By Lubin-Tate theory (see \S 1 of \cite{LT65}) there exists a formal group $F$ such that $[p(1+p^n)](X) = f(X)$. This group is attached to the uniformizer $p(1+p^n)$ of $\Qp$. Likewise, $\Gm$ is attached to $p$. The formal group $F$ is not isomorphic to $\Gm$ over $\Qp$ as $p \neq p(1+p^n)$ and any Lubin-Tate group attached to a uniformizer $\pi$ determines $\pi$. 

However, we have $f(\zeta_p-1)=0$ and $f(\zeta_{p^k}-1)=\zeta_{p^{k-1}}-1$ for all $k \leq n$, so that $\zeta_{p^k}-1 \in \Tors(F)$ for all $k \leq n$. This proves the proposition.
\end{proof}

If $\Tors(F) \cap \Tors(G)$ is large, then are $F$ and $G$  close to each other in some sense?

\subsection{The logarithm of a formal group}
Using the logarithms of formals groups, we can give a very short proof of a weaker form of theorem A, namely: if $\Tors(F)=\Tors(G)$ (and this common set is infinite), then $F=G$. Indeed, $\Log_F$ is holomorphic on $\MM_{\Cp}$ and its zeroes are precisely the elements of $\Tors(F)$, with multiplicity $1$. In addition, $\Log_F'$ is a bounded power series since $\mathrm{d} \Log_F$ is the normalized invariant differential on $F$. If $\Tors(F)=\Tors(G)$, then $\Log_F$ and $\Log_G$ have the same zeroes, so that they differ by a unit $u$. A unit is necessarily bounded. We have $\Log_G = u \cdot \Log_F$ and hence $\Log_G' = u \cdot \Log_F' + u' \cdot \Log_F$. Since $\Log_G'$ and $\Log_F'$ and $u$ are bounded, but not $\Log_F$, we must have $u'=0$ (the sup norms $\|{\cdot}\|_r$ on circles are multiplicative). This implies that $u \in \OO_K^\times$ and then that $u=1$ since $\Log_F'(0)=\Log_G'(0)=1$, so that $\Log_F = \Log_G$ and $F=G$. The same argument gives the following characterization of the logarithm of a formal group of finite height.

\begin{prop}
\label{uniqlog}
If $F$ is a formal group of finite height, then the power series $\Log_F$ is the unique element of $X +X^2 \cdot K\dcroc{X}$ that is holomorphic on $\MM_{\Cp}$, whose zero set is precisely $\Tors(F)$, with multiplicity $1$, and whose derivative is bounded.
\end{prop}

\subsection{More rigidity}
A common generalization of theorems A and B would be the assertion that if a power series $h$ maps infinitely many torsion points of $F$ to torsion points of $G$, then $h \in \Hom(F,G)$. In order to prove this using the same method as in the proof of theorem B, we would need to show that there exists $\sigma \in \Gal(\Kalg/K)$ that acts on $\Tors(F)$ and $\Tors(G)$ by two power series $u_F$ and $u_G$, satisfying some stability condition. If $G$ is a  Lubin-Tate formal group (for some finite extension of $\Qp$ contained in $K$), there is a character $\chi_G : \Gal(\Kalg/K) \to \OO_K^\times$ such that $\sigma(z)=[\chi_G(\sigma)](z)$ for all $z \in \Tors(G)$ (theorem 2 of \cite{LT65}).

\begin{theo}
\label{thBter}
If $F$ is a formal group and $G$ is a Lubin-Tate formal group, both defined over $\OO_K$, and if $h(X) \in X \cdot \OO_K \dcroc{X}$ is such that $h'(0) \neq 0$ and $h(z) \in \Tors(G)$ for infinitely many $z \in \Tors(F)$, then $h \in \Hom(F,G)$.
\end{theo}

\begin{proof}
Since $\Tors(F)$ is infinite, $F$ is of finite height. By corollary \ref{scalopen}, there exists an element $\sigma \in \Gal(\Kalg/K)$ and a stable endomorphism $u_F$ of $F$ such that $\sigma(z_F) = u_F(z_F)$ if $z_F \in \Tors(F)$. Let $u_G(X) = [\chi_G(\sigma)](X)$, so that $\sigma(z_G) = u_G(z_G)$ if $z_G \in \Tors(G)$. 

If $z \in \Tors(F)$ is such that $h(z) \in \Tors(G)$, then $\sigma(h(z))=u_G(h(z))$ and $\sigma(h(z)) = h (\sigma(z)) = h(u_F(z))$. The power series $u_G \circ h - h \circ u_F$ therefore vanishes at infinitely many points of $\MM_{\Cp}$, so that $u_G \circ h = h \circ u_F$. Since $h'(0) \neq 0$, we have $u_F'(0) = u_G'(0)$ and $u_G$ is stable. The theorem now follows from lemma \ref{ishomo} below.
\end{proof}

\begin{lemm}
\label{ishomo}
Let $F$ and $G$ be two formal groups and let $f$ and $g$ be endomorphisms of $F$ and $G$, with $g$ stable. If $h(X) \in X \cdot \OO_K \dcroc{X}$ is such that $h'(0) \neq 0 $ and $h \circ f = g \circ h$, then $h \in \Hom(F,G)$.
\end{lemm}

\begin{proof}
Consider the power series $K(X,Y) = h \circ F (h^{\circ -1}(X),h^{\circ -1}(Y))$. We have 
\[ K \circ g = h \circ F \circ h^{\circ -1} \circ g = h \circ F \circ f \circ h^{\circ -1} = h \circ f \circ F \circ h^{\circ -1} = g \circ h \circ F \circ h^{\circ -1}  = g \circ K \]
Since $K$ and $G$ commute with $g$ and have the same derivatives at $0$, we have $K=G$ by theorem \ref{comgen} and hence $h \circ F = G \circ h$, so that $h \in \Hom(F,G)$.
\end{proof}

Note that the hypothesis of the lemma imply that $f'(0)=g'(0)$ so that if one series is stable, then both are.

\subsection{Homotheties and stable $p$-adic dynamical systems}

If $F$ is a formal group of finite height, then $\End(F)$ is a set of power series that commute with each other under composition. One can forget about the formal group and study certain sets $\Dyn$ of elements of $X \cdot \OO_K \dcroc{X}$ that commute with each other under composition. This is the object of Lubin's theory of $p$-adic dynamical systems (see \cite{L94}). 

Let us say that $\Dyn \subset X \cdot \OO_K \dcroc{X}$ is a stable $p$-adic dynamical system of finite height if the elements of $\Dyn$ commute with each other under composition, and if $\Dyn$ contains a stable series $f$ such that $f'(0) \in \MM_K$ and $f(X) \not\equiv 0 \bmod{\MM_K}$ 
(i.e.\  $f$ is of finite height) as well as a stable series $u$ such that $u'(0) \in \OO_K^\times$. We can then assume that $\Dyn$ is as large as possible, namely that any power series $g \in X \cdot \OO_K \dcroc{X}$ that commutes with the elements of $\Dyn$ belongs to $\Dyn$. For example, if $F$ is a formal group of finite height, then $\End(F)$ is a stable $p$-adic dynamical system.

Given a stable $p$-adic dynamical system of finite height $\Dyn$, the set $\preper(g)$ is independent of the choice of a stable $g \in \Dyn$ (see \S 3 of \cite{L94}). One can then define $\preper(\Dyn)$ as the preperiodic set of any stable element of $\Dyn$. To what extent does $\preper(\Dyn)$ determine a stable $p$-adic dynamical system of finite height $\Dyn$?

In order to extend our results from formal groups to stable $p$-adic dynamical systems of finite height, we can ask whether the consequence of corollary \ref{scalopen} holds in more generality: for which stable $p$-adic dynamical systems of finite height $\Dyn$ is there a stable power series $w \in \Dyn$ and an element $\sigma \in \Gal(\Kalg/K)$ such that $\sigma(z)=w(z)$ for all $z \in \preper(\Dyn)$?

\providecommand{\bysame}{\leavevmode ---\ }
\providecommand{\og}{``}
\providecommand{\fg}{''}
\providecommand{\smfandname}{\&}
\providecommand{\smfedsname}{\'eds.}
\providecommand{\smfedname}{\'ed.}
\providecommand{\smfmastersthesisname}{M\'emoire}
\providecommand{\smfphdthesisname}{Th\`ese}


\begin{thebibliography}{BFT18}

\bibitem[BD11]{BDM11}
{\scshape M.~Baker {\normalfont \smfandname} L.~DeMarco} -- {\og Preperiodic
  points and unlikely intersections\fg}, \emph{Duke Math. J.} \textbf{159}
  (2011), no.~1, p.~1--29.

\bibitem[BE87]{BE87}
{\scshape I.~N. Baker {\normalfont \smfandname} A.~Er\"{e}menko} -- {\og A
  problem on {J}ulia sets\fg}, \emph{Ann. Acad. Sci. Fenn. Ser. A I Math.}
  \textbf{12} (1987), no.~2, p.~229--236.

\bibitem[BFT18]{BFT}
{\scshape F.~Bogomolov, H.~Fu {\normalfont \smfandname} Y.~Tschinkel} -- {\og
  Torsion of elliptic curves and unlikely intersections\fg}, in \emph{Geometry
  and Physics: Volume I. A Festschrift in honour of Nigel Hitchin}, Oxford
  University Press, 2018, p.~19--38.

\bibitem[BT07]{BT07}
{\scshape F.~Bogomolov {\normalfont \smfandname} Y.~Tschinkel} -- {\og
  Algebraic varieties over small fields\fg}, in \emph{Diophantine geometry},
  CRM Series, vol.~4, Ed. Norm., Pisa, 2007, p.~73--91.

\bibitem[Cha08]{C08}
{\scshape C.-L. Chai} -- {\og A rigidity result for {$p$}-divisible formal
  groups\fg}, \emph{Asian J. Math.} \textbf{12} (2008), no.~2, p.~193--202.

\bibitem[DKY]{DKY}
{\scshape L.~DeMarco, H.~Krieger {\normalfont \smfandname} H.~Ye} -- {\og Lower
  bounds on dynamical height pairings\fg}, in preparation.

\bibitem[Hid14]{H14}
{\scshape H.~Hida} -- {\og Hecke fields of {H}ilbert modular analytic
  families\fg}, in \emph{Automorphic forms and related geometry: assessing the
  legacy of {I}. {I}. {P}iatetski-{S}hapiro}, Contemp. Math., vol. 614, Amer.
  Math. Soc., Providence, RI, 2014, p.~97--137.

\bibitem[LT65]{LT65}
{\scshape J.~Lubin {\normalfont \smfandname} J.~Tate} -- {\og Formal complex
  multiplication in local fields\fg}, \emph{Ann. of Math. (2)} \textbf{81}
  (1965), p.~380--387.

\bibitem[Lub64]{L64}
{\scshape J.~Lubin} -- {\og One-parameter formal {L}ie groups over
  {${\mathfrak{p}}$}-adic integer rings\fg}, \emph{Ann. of Math. (2)}
  \textbf{80} (1964), p.~464--484.

\bibitem[Lub67]{L67}
\bysame , {\og Finite subgroups and isogenies of one-parameter formal {L}ie
  groups\fg}, \emph{Ann. of Math. (2)} \textbf{85} (1967), p.~296--302.

\bibitem[Lub94]{L94}
\bysame , {\og Nonarchimedean dynamical systems\fg}, \emph{Compositio Math.}
  \textbf{94} (1994), no.~3, p.~321--346.

\bibitem[Sen73]{S73}
{\scshape S.~Sen} -- {\og Lie algebras of {G}alois groups arising from
  {H}odge-{T}ate modules\fg}, \emph{Ann. of Math. (2)} \textbf{97} (1973),
  p.~160--170.

\bibitem[Ser67]{S67}
{\scshape J.-P. Serre} -- {\og Sur les groupes de {G}alois attach\'{e}s aux
  groupes {$p$}-divisibles\fg}, in \emph{Proc. {C}onf. {L}ocal {F}ields
  ({D}riebergen, 1966)}, Springer, Berlin, 1967, p.~118--131.

\bibitem[Ser18]{S16}
{\scshape V.~Serban} -- {\og An infinitesimal $p$-adic multiplicative
  {M}anin-{M}umford conjecture\fg}, \emph{J. Th\'{e}or. Nombres Bordeaux}
  \textbf{30} (2018), no.~2, p.~393--408.

\bibitem[SS95]{SS95}
{\scshape W.~Schmidt {\normalfont \smfandname} N.~Steinmetz} -- {\og The
  polynomials associated with a {J}ulia set\fg}, \emph{Bull. London Math. Soc.}
  \textbf{27} (1995), no.~3, p.~239--241.

\end{thebibliography}
\end{document}